\documentclass[oneside,12pt]{amsart}
\usepackage{amsmath,amsfonts, latexsym,amssymb}
\usepackage[active]{srcltx}
\newtheorem{theorem}{Theorem}[section]
\newtheorem{lemma}[theorem]{Lemma}
\newtheorem{corollary}[theorem]{Corollary}

\theoremstyle{definition}

\newtheorem{remark}[theorem]{Remark}
\numberwithin{equation}{section}


\begin{document}

\title{On the  mod $p^7$ determination of ${2p-1\choose p-1}$}

\author{Romeo Me\v strovi\' c}

\address{Maritime Faculty, University of Montenegro, Dobrota 36,
 85330 Kotor, Montenegro} \email{romeo@ac.me}

{\renewcommand{\thefootnote}{}
\setcounter{footnote}{0}}

\maketitle

\begin{abstract}
In this paper we prove 
that for any prime $p\ge 11$  holds
  $$
{2p-1\choose p-1}\equiv 1 -2p \sum_{k=1}^{p-1}\frac{1}{k}
+4p^2\sum_{1\le i<j\le p-1}\frac{1}{ij}\pmod{p^7}. 
 $$
This  is a generalization of the famous Wolstenholme's theorem
which asserts that ${2p-1\choose p-1} \equiv 1 \,\,(\bmod\,\,p^3)$
for all primes $p\ge 5$.
Our proof is elementary  and it does not use
a standard technique involving the 
classic formula for the  power sums in terms of the Bernoulli numbers.
Notice that the above congruence reduced modulo $p^6$,
$p^5$ and $p^4$ yields related congruences 
obtained by  R. Tauraso, J. Zhao and J.W.L. Glaisher, respectively.
\end{abstract}

\section{Introduction and Statement  of  Results}

{\it Wolstenholme's theorem} (e.g. see  \cite{W}, \cite{Gr})
asserts that if $p$ is a prime greater than 3, then 
the binomial coefficient
${2p-1\choose p-1}$ satisfies the congruence
 \begin{equation}\label{cong1.1}
{2p-1\choose p-1} \equiv 1 \pmod{p^3}
   \end{equation}
for any prime $p\ge 5$.
 It is well known  (e.g. see \cite[p. 89]{HW}) that
this theorem is equivalent to the assertion 
 that the numerator of the fraction $1+\frac{1}{2}+ 
\frac{1}{3}+\cdots +\frac{1}{p-1}$ is divisible by $p^2$ for
any prime $p\ge 5$.

Further, by a special case of 
Glaisher's congruence (\cite[p. 21]{Gl1}, \cite[p. 323]{Gl2};   
also cf. \cite[Theorem 2]{M}), for any prime $p\ge 5$ we have 
    \begin{equation}\label{cong1.2}
{2p-1\choose p-1}\equiv 1-2p \sum_{k=1}^{p-1}\frac{1}{k}\equiv
1 -\frac{2p^3}{3}B_{p-3}\pmod{p^4},
      \end{equation}
where $B_k$ is the $k$th Bernoulli number. 
A. Granville \cite{Gr} established broader generalizations of 
Wolstenholme's theorem. 
More recently, C. Helou and G. Terjanian \cite{HT}  established many  
Wolstenholme's type congruences modulo $p^k$ with a prime $p$
and $k\in\{4,5,6\}$. 
One of their main results 
\cite[Proposition 2, pp. 488-489]{HT} 
is a congruence of the form ${np\choose mp}\equiv f(n,m,p) {n\choose m}
\,\,(\bmod\,\, p^6)$,
where $p\ge 3$ is a prime number, $m,n,\in \Bbb N$ with 
$0\le m\le n$, and $f$ is the function on $m,n$ and $p$ involving 
Bernoulli numbers $B_k$ $(k\in\Bbb N)$. 
In particular, for $p\ge 5$, $m=1$ and $n=2$,
using the fact that $\frac{1}{2}{2p\choose p}={2p-1\choose p-1}$,
this congruence  yields \cite[Corollary 1]{HT}
   \begin{equation}\label{congr1.3}  
{2p-1\choose p-1} \equiv 1-p^3B_{p^3-p^2-2}
+\frac{p^5}{3}B_{p-3}-\frac{6p^5}{5}B_{p-5}\pmod{p^6}.
      \end{equation}

Recently, R. Tauraso \cite[Theorem 2.4]{T} proved that for any prime
$p>5$
   \begin{equation*}  
{2p-1\choose p-1}\equiv 1 +2p\sum_{k=1}^{p-1}\frac{1}{k}
+\frac{2p^3}{3}\sum_{k=1}^{p-1}\frac{1}{k^3}\pmod{p^6}. 
    \end{equation*} 

In this paper we improve the above congruence as follows.

\begin{theorem}
 Let $p\ge 11$  be a prime. Then
   \begin{equation}\label{cong1.4}
{2p-1\choose p-1}\equiv 1 -2p\sum_{k=1}^{p-1}\frac{1}{k}
+4p^2\sum_{1\le i<j\le p-1}\frac{1}{ij}\pmod{p^7}. 
     \end{equation} 
\end{theorem}

\begin{remark}
Note that the congruence \eqref{cong1.4} for $p=3$ and $p=5$ reduces 
to the identity, while for $p=7$ \eqref{cong1.4} is satisfied modulo   $7^6$.
\end{remark}

Applying a technique of Helou and Terjanian \cite{HT} 
based on Kummer type congruences, the congruence 
\eqref{cong1.4}  may be expressed in terms of  the Bernoulli numbers
as follows.
   \begin{corollary}
Let $p\ge 11$ be a  prime. Then
   \begin{eqnarray}    
{2p-1\choose p-1} &\equiv& 1-p^3B_{p^4-p^3-2}
+p^5\left(\frac{1}{2}B_{p^2-p-4}-2B_{p^4-p^3-4}\right)\nonumber\\
&&+p^6\left(\frac{2}{9}B_{p-3}^2-\frac{1}{3}B_{p-3}-\frac{1}{10}B_{p-5}\right)
\pmod{p^7}.\label{cong1.5}
  \end{eqnarray}    
\end{corollary}
Note that reducing the moduli, and using 
the Kummer congruences.from \eqref{cong1.5} can be easily deduced 
the congruence \eqref{congr1.3}. 
\begin{corollary}$($cf. \cite[Theorem 2.4]{T}$).$ 
Let $p\ge 7$  be a prime. Then
    \begin{equation*}  
{2p-1\choose p-1}\equiv 1 -2p\sum_{k=1}^{p-1}\frac{1}{k}
-2p^2\sum_{k=1}^{p-1}\frac{1}{k^2}\equiv 1 +2p\sum_{k=1}^{p-1}\frac{1}{k}
+\frac{2p^3}{3}\sum_{k=1}^{p-1}\frac{1}{k^3}\pmod{p^6}. 
    \end{equation*} 
\end{corollary}

\begin{corollary} 
$($\cite[Theorem 3.2]{Z2}, \cite[p. 385]{M}$).$ 
 Let $p\ge 7$  be a prime. Then
  \begin{equation*}
{2p-1\choose p-1}\equiv 1 +2p\sum_{k=1}^{p-1}\frac{1}{k}
\equiv 1-p^2\sum_{k=1}^{p-1}\frac{1}{k^2}\pmod{p^5}. 
    \end{equation*} 
\end{corollary}

A prime $p$ is said to be a {\it Wolstenholme prime} if it 
satisfies the congruence 
$
{2p-1\choose p-1} \equiv 1 \,\,(\bmod{\,\,p^4}).$
By the  congruence (1.2) we see that  a prime $p$
is a Wolstenholme prime if and only if 
$p$ divides  the numerator of $B_{p-3}$.
The two known such primes are 16843 and 2124679, and 
recently, R.J. McIntosh and E.L. Roettger \cite{MR} reported  that these
primes are only two  Wolstenholme primes less than $10^9$.
However, by using the argument based on the prime number theorem, 
McIntosh  \cite[p. 387]{M}  conjectured that there are infinitely many 
Wolstenholme primes, and that no  prime satisfies
the congruence ${2p-1\choose p-1} \equiv 1\,\, (\bmod{\,\,p^5})$.

\begin{remark} In \cite[Corollary 1]{Me} the author proved that for any 
Wolstenholme prime $p$ holds
    \begin{eqnarray}
{2p-1\choose p-1} &\equiv& 1 -2p \sum_{k=1}^{p-1}\frac{1}{k}
-2p^2\sum_{k=1}^{p-1}\frac{1}{k^2}\nonumber\\
&\equiv& 1 +2p \sum_{k=1}^{p-1}\frac{1}{k}
+\frac{2p^3}{3}\sum_{k=1}^{p-1}\frac{1}{k^3}\pmod{p^7},\label{cong1.6}
    \end{eqnarray}  
and conjectured  \cite[Remark 1]{Me} that
 any of the previous congruences for a prime $p$ yields that $p$ is necessarily 
a Wolstenholme prime. Note that this conjecture concerning the first 
above congruence may be confirmed
by using our congruence \eqref{cong1.4}. Namely, if a prime $p$
satisfied the first congruence of \eqref{cong1.6}, then by  \eqref{cong1.4}
must be
  \begin{eqnarray}\label{cong1.7}
{2p-1\choose p-1}
&\equiv& 1 -2p \sum_{k=1}^{p-1}\frac{1}{k}-2p^2\sum_{k=1}^{p-1}\frac{1}{k^2}
\nonumber\\
&\equiv & 1 -2p\sum_{k=1}^{p-1}\frac{1}{k}+
4p^2\sum_{1\le i<j\le p-1}\frac{1}{ij}\pmod{p^7}. \label{cong1.7}
  \end{eqnarray} 
Using the identity 
   $$ 
2\sum_{1\le i<j\le p-1}\frac{1}{ij}=
\left(\sum_{k=1}^{p-1}\frac{1}{k}\right)^2-\sum_{k=1}^{p-1}\frac{1}{k^2},
  $$
the second congruence in \eqref{cong1.7} immediately reduces to
     $$
 2p^2\left(\sum_{k=1}^{p-1}\frac{1}{k}\right)^2\equiv 0\pmod{p^7},
  $$
whence it follows that
   $$
\sum_{k=1}^{p-1}\frac{1}{k}\equiv 0\pmod{p^3}.
  $$
Finally, substituting this into the first Glaisher's 
congruence in \eqref{cong1.2}, we find that
  $$
{2p-1\choose p-1}\equiv 0\pmod{p^4}.
 $$ 
Hence, $p$ must be a Wolstenholme prime, and so, our conjecture 
is confirmed related to the  first congruence of \eqref{cong1.6}. 

The situation is more complicated in relation to
the conjecture concerning the second congruence of \eqref{cong1.6}.
Then comparing this congruence and \eqref{cong1.4}, as in the previous case
we obtain
  $$
 2\sum_{k=1}^{p-1}\frac{1}{k}-p\left(\sum_{k=1}^{p-1}\frac{1}{k}\right)^2
+p\sum_{k=1}^{p-1}\frac{1}{k^2}+\frac{p^2}{3}
\sum_{k=1}^{p-1}\frac{1}{k^3}\equiv 0\pmod{p^6}.
  $$
However, from the above congruence we are unable to deduce 
that $p$ must be a Wolstenholme prime.
   \end{remark}

\begin{remark}    
It follows from Corollary 1.5 that 
 $p^3\mid \sum_{k=1}^{p-1}1/k$
and $p^2\mid \sum_{k=1}^{p-1}1/k^2$ for any Wolstenholme prime $p$.
This argument together with a technique 
applied in the proof of  Theorem 1.1
suggests the conjecture that such a prime $p$
satisfies the congruence \eqref{cong1.4} modulo $p^8$. However, 
a direct calculation shows that this is not true
for the Wolstenholme prime 16843.

 As noticed in Remark 1.2,  
the congruence \eqref{cong1.4} for $p=3$ and $p=5$ reduces to the identity. 
However, our computation via Mathematica shows that 
no prime in the range $7\le p<500000$ satisfies 
the congruence \eqref{cong1.4} with the modulus $p^8$ instead of $p^7$.  
Nevertheless, using the  heuristic argument 
for the "probability" that a prime $p$ satisfies \eqref{cong1.4} modulo 
$p^8$ is about $1/p$, 
we conjecture that there are infinitely many primes satisfying \eqref{cong1.4} 
modulo $p^8$.
\end{remark}
\section{Proof of  Theorem 1.1 and Corollaries 1.4 and 1.5} 

For the proof of  Theorem 1.1,  we will need some elementary 
auxiliary results.

For a prime $p\ge 3$  and a positive integer
$n\le p-2$  we denote 
  $$
R_n(p):=\sum_{i=1}^{p-1}\frac{1}{k^n} \,\,\,\,{\rm and}\,\,\,\,
H_n(p):=\sum_{1\le i_1<i_2<\cdots<i_n\le p-1}\frac{1}{i_1i_2\cdots i_n},
  $$
with the convention  that $H_1(p)=R_1(p)$. 
In the sequel we shall often write
throughout proofs  $R_n$ and $H_n$ instead of 
$R_n(p)$ and $H_n(p)$, respectively.

Observe that by Wolstenholme's theorem, $p^2\mid R_1(p)$ for any prime
$p\ge 5$, which can be generalized as follows.

\begin{lemma}  {\rm (\cite[Theorem 3]{B}; also 
see   \cite{ZC}  or \cite[Theorem 1.6]{Z1}).}  For any
prime $p\ge 5$  and a positive integer $n\le p-3$, we have 
 $$
R_n(p)\equiv 0  \pmod{p^2}\quad
if\,\,\, 2 \nmid n,\quad  and\quad R_n(p)\equiv 0  \pmod{p}\quad
if\,\,\, 2\mid n.
 $$ 
\end{lemma}

\begin{lemma}
 For any prime $p\ge 7$, 
we have 
     \begin{equation}\label{cong2.1}    
H_3(p)\equiv \frac{R_3(p)}{3} -\frac{R_1(p)R_2(p)}{2}\pmod{p^6} 
  \end{equation}
and
    \begin{equation}\label{cong2.2}
H_4(p)\equiv -\frac{R_4(p)}{4} +\frac{(R_2(p))^2}{8}\pmod{p^4}.
 \end{equation}
In particular, $p^2\mid H_3(p)$, $p\mid H_2(p)$ and $p\mid H_4(p)$.
\end{lemma}

\begin{proof} 
Substituting  the shuffle relation $H_2=(R_1^2-R_2)/2$
into the identity $3H_3=R_3-R_1R_2+H_2R_1$, we find that
$H_3=\frac{R_3}{3}-\frac{R_1R_2}{2}+\frac{R_1^3}{6}$.
This equality together with the fact that $p^2\mid R_1$
yields the congruence (2.1), and thus $p^2\mid H_3$.

Similarly, by  Newton's formula \cite{J}, we have the identity
  $$
4H_4=-R_4+H_1R_3-H_2R_2+H_3R_1.
  $$
Since by Lemma 2.1, $p^4\mid R_1R_3=H_1R_3$,
and since $p^2\mid H_3$ we also have $p^4\mid H_3R_1$.
Substituting this and   $H_2=(R_1^2-R_2)/2$
into the above  identity,  we obtain
  $$
4H_4\equiv -R_4-\frac{R_1^2R_2}{2}+\frac{R_2^2}{2}\pmod{p^4}.
  $$
Since by Lemma 2.1, $p^5\mid R_1^2R_2$,
we can exclude the term $R_1^2R_2/2$ in the above congruence
to obtain (2.2), and so $p\mid H_4$. This completes the proof. 
\end{proof}

\begin{lemma} For any prime $p$  and any positive
integer $r$, we have
      \begin{equation}\label{cong2.3}
2R_1\equiv-\sum_{i=1}^rp^iR_{i+1}\pmod{p^{r+1}}.
 \end{equation}
\end{lemma}

\begin{proof}
Multiplying   the identity
  $$
1+\frac{p}{i}+\cdots+\frac{p^{r-1}}{i^{r-1}}=\frac{p^r-i^r}{i^{r-1}(p-i)}
 $$
by $-p/i^2$ ($1\le i\le p-1$), we obtain
  $$   
-\frac{p}{i^2}\left(1+\frac{p}{i}+\cdots+\frac{p^{r-1}}{i^{r-1}}\right)
=\frac{-p^{r+1}+pi^r}{i^{r+1}(p-i)}
\equiv \frac{p}{i(p-i)}\pmod{p^{r+1}}.
     $$   
Therefore
 $$
\left(\frac{1}{i}+\frac{1}{p-i}\right)\equiv
-\left(\frac{p}{i^2}+\frac{p^2}{i^3}+\cdots+\frac{p^{r}}{i^{r+1}}\right)
\pmod{p^{r+1}},
 $$
whence after summation over $i=1,\ldots , p-1$ we immediately obtain 
(2.3).  This concludes the proof.
\end{proof}

\begin{lemma} 
For any prime $p\ge 7$  we have
\begin{equation*}
2R_{1}(p)\equiv -pR_{2}(p)\pmod{p^4},
 \end{equation*}
and for any prime $p\ge 11$ holds
     \begin{equation*}
2R_{3}(p)\equiv -3pR_{4}(p)\pmod{p^4}.
 \end{equation*}
  \end{lemma}

\begin{proof} 
Note that  by Lemma 2.3,  
 $$
2R_1\equiv -pR_2-p^2R_3-p^3R_4 \pmod{p^4}.
 $$
Since by Lemma 2.1, $p^2\mid R_3$ and $p\mid R_4$ for any prime $p\ge 7$, 
the above congruence reduces to  the first congruence in our lemma.
     
Since for each $1\le k\le p-1$ 
  $$
\frac{1}{k^3}+\frac{1}{(p-k)^3}=\frac{p^3-3p^2k+3pk^2}{k^3(p-k)^3},
  $$
it follows that
    \begin{equation}\label{cong2.4}\begin{split}
2R_3 &= \sum_{k=1}^{p-1}\left(\frac{1}{k^3}+\frac{1}{(p-k)^3}\right)\\
&=p^3\sum_{k=1}^{p-1}\frac{1}{k^3(p-k)^3}-
3p^2\sum_{k=1}^{p-1}\frac{1}{k^2(p-k)^3}+3p
\sum_{k=1}^{p-1}\frac{1}{k(p-k)^3}.
 \end{split}\end{equation}
First observe that, applying Lemma 2.1, for each prime $p\ge 11$
we have 
 \begin{equation}\label{cong2.5}\begin{split}
\sum_{k=1}^{p-1}\frac{1}{k^3(p-k)^3}\equiv -\sum_{k=1}^{p-1}\frac{1}{k^6}
\equiv 0\pmod{p}.
  \end{split}\end{equation}
Further, in view of the fact that 
$1/(p-k)\equiv -(p+k)/k^2\,\,(\bmod\,\, p^2)$, 
and that for each prime $p\ge 11$, 
$p\mid R_6$ and $p^2\mid R_5$ by  Lemma 2.1,  we have
 \begin{equation}\label{cong2.6}\begin{split}
 \sum_{k=1}^{p-1}\frac{1}{k^2(p-k)^3}&=
\sum_{k=1}^{p-1}\frac{1}{(p-k)^2k^3}\\
&\equiv \sum_{k=1}^{p-1}\frac{(p+k)^2}{k^7}\quad\,\,\qquad \pmod{p^2}\\
&\equiv \sum_{k=1}^{p-1}\frac{2p}{k^6}+\sum_{k=1}^{p-1}\frac{1}{k^5}
\equiv 0\pmod{p^2}.
 \end{split}\end{equation}
Substituting (2.5) and (2.6) into (2.4), 
we get
     \begin{equation}\label{cong2.7}\begin{split}
2R_3\equiv 3p\sum_{k=1}^{p-1}\frac{1}{k(p-k)^3}\pmod{p^4}.
 \end{split}\end{equation}
Next from the identity
  $$
\frac{1}{k(p-k)^3}+\frac{1}{k^4}=
\frac{p^3}{k^4(p-k)^3}-\frac{3p^2}{k^3(p-k)^3}
+\frac{3p}{k^2(p-k)^3},
  $$
for $k=1,2,\ldots, p-1$, we obtain
 $$
\frac{1}{k(p-k)^3}+\frac{1}{k^4}\equiv 
\frac{3p^2}{k^6}
+\frac{3p}{k^2(p-k)^3}\pmod{p^3}.
  $$
After summation over $k=1,\ldots , p-1$, the above congruence
gives
  $$
\sum_{k=1}^{p-1}\frac{1}{k(p-k)^3}+R_4\equiv 
3p^2R_6 +3p\sum_{k=1}^{p-1}\frac{1}{k^2(p-k)^3}\pmod{p^3}.
  $$
Since by Lemma 2.1, $p\mid R_6$ for any prime $p\ge 11$, substituting this  
and (2.6) into the above congruence, we obtain
 $$
\sum_{k=1}^{p-1}\frac{1}{k(p-k)^3}\equiv -R_4
\pmod{p^3}.
  $$
Substituting this into (2.7), we finally obtain
  $$
2R_3\equiv -3pR_4\pmod{p^4}.
  $$
This completes the proof.
 \end{proof}

\begin{proof}[Proof of Theorem 1.1.] 
For any prime $p\ge 11$, we have
\begin{eqnarray*}   
{2p-1\choose p-1}&=&\frac{(p+1)(p+2)\cdots(p+k)\cdots (p+(p-1))}{1\cdot 2
\cdots k\cdots p-1}\\
&=& \left(\frac{p}{1}+1\right)\left(\frac{p}{2}+1\right)\cdots
\left(\frac{p}{k}+1\right)
\cdots\left(\frac{p}{p-1}+1\right)\\
&=&1+\sum_{i=1}^{p-1}\frac{p}{i}+
\sum_{1\le i_1<i_2\le p-1}\frac{p^2}{i_1i_2}+\cdots+
\sum_{1\le i_1<i_2<\cdots <i_k\le p-1}\frac{p^k}{i_1i_2\cdots i_k}\\
&&+\cdots+\frac{p^{p-1}}{(p-1)!}=1+\sum_{k=1}^{p-1}p^{k}H_k=
1+\sum_{k=1}^{6}p^{k}H_k+\sum_{k=7}^{p-1}p^{k}H_k.
 \end{eqnarray*}
By Lemmas 2.1 and 2.2, we have $R_1\equiv R_3\equiv R_5\equiv H_3\equiv
0 \,\,(\bmod \,\,p^2)$ and 
$R_2\equiv R_4\equiv R_6\equiv H_2\equiv H_4\equiv 0 \,\,(\bmod \,\,p)$
for any prime $p\ge 11$.
Since by Newton's formula,  $5H_5=R_5+\sum_{i=1}^4(-1)^iH_iR_{5-i}$
and $6H_6=-R_6-\sum_{i=1}^5(-1)^iH_iR_{6-i}$, it follows 
from the previous congruences that $p^2\mid H_5$ and $p\mid H_6$.
  Therefore, $p^7\mid \sum_{k=5}^{p-1}p^{k}H_k$ 
for any prime $p\ge 11$, and so the above expansion yields
      \begin{equation}\label{cong2.8}
{2p-1\choose p-1}\equiv 1+pH_1+p^2H_2+p^3H_3+p^4H_4
\pmod{p^7}.
    \end{equation}
Recall that $H_1=R_1$ and $H_2=(R_1^2-R_2)/2$.
The congruences from Lemma 2.2 yield
$H_3\equiv\frac{R_3}{3} -\frac{R_1R_2}{2}
\,\,(\bmod{\,\,p^4})$ and $H_4\equiv -\frac{R_4}{4} +
\frac{R_2^2}{8}\,\,(\bmod{\,\,p^3})$.
Substituting  all the previous  
expressions for $H_i$, $i=1,2,3,4$, into (2.8), 
we find that
    \begin{equation}\label{cong2.9}\begin{split}
{2p-1\choose p-1}\equiv & 1+pR_1+\frac{p^2}{2}(R_1^2-R_2)\\
&+\frac{p^3}{6}(2R_3-3R_1R_2)+\frac{p^4}{8}(R_2^2-2R_4)
\pmod{p^7}.  \end{split}
   \end{equation}
Further, by  Lemma 2.4, we have 
      \begin{equation}\label{cong2.10}
2R_{1}\equiv -pR_{2}\pmod{p^4}
 \end{equation}
and 
     \begin{equation}\label{cong2.11}
2R_{3}\equiv -3pR_{4}\pmod{p^4}.
 \end{equation}
The congruences (2.10)  and (2.11) yield $p^4R_2^2\equiv -2p^3R_1R_2
\,\,(\bmod{\,\,p^7})$
and $p^4R_4\equiv -\frac{2}{3}p^3R_3\,\,(\bmod{\,\,p^7})$,
respectively. Substituting these  congruences
into the last term on the right hand side of (2.9),
we obtain
    \begin{equation}\label{cong2.12}\begin{split}
{2p-1\choose p-1}\equiv & 1+pR_1+\frac{p^2}{2}(R_1^2-R_2)\\
&-\frac{3p^3}{4}R_1R_2+\frac{p^3}{2}R_3\pmod{p^7}.
\end{split}\end{equation}
It remains to eliminate $R_3$ from (2.12).  
Note that  by Lemma 2.3, 
 $2R_1\equiv -pR_2-p^2R_3-p^3R_4-p^4R_5-p^5R_6\,\,(\bmod{\,\,p^6})$.
Since by Lemma 2.2, $p^2\mid R_5$ and $p\mid R_6$, 
the previous congruence reduces to  
      \begin{equation}\label{cong2.13}
2R_1\equiv -pR_2-p^2R_3-p^3R_4\pmod{p^6}.
   \end{equation}
We use again the congruence (2.11) in the form 
$p^3R_4\equiv -\frac{2}{3}p^2R_3\,\,(\bmod{\,\,p^6})$,
which by inserting in (2.13) yields
$2R_1\equiv -pR_2-\frac{1}{3}p^2R_3\,\,(\bmod{\,\,p^6})$.
Multipying by $3p$, this implies  
  \begin{equation}\label{cong2.14}
p^3R_3\equiv -6pR_1-3p^2R_2\pmod{p^7}.
  \end{equation}
Substituting this into the last term of (2.13), 
we immediately get
      \begin{equation}\label{cong2.15}
{2p-1\choose p-1}\equiv 1-2pR_1-2p^2R_2+\frac{p^2}{4}R_1(2R_1-3pR_2)
\pmod{p^7}. 
\end{equation}
Now we write (2.10) as
 $$
2R_1-3pR_2\equiv 8R_1\pmod{p^4}.
  $$
Since $p^2\mid R_1$, and so $p^4\mid p^2R_1$,
multiplying the above congruence by
$\frac{1}{4}p^2R_1$, we find that 
 $$
\frac{p^2}{4}R_1(2R_1-3pR_2)\equiv 2p^2R_1^2\pmod{p^7}.
  $$
Replacing this into (2.15), we  obtain 
  \begin{equation}\label{cong2.16}
{2p-1\choose p-1}\equiv  1-2pR_1+2p^2(R_1^2-R_2)\pmod{p^7}.  
\end{equation}
which by the identity $(R_1^2-R_2)/2=H_2$
yields the desired congruence. This completes the proof.
\end{proof}

\begin{proof}[Proof of Corollary 1.4.] The first 
congruence in Corollary 1.4
for $p\ge 11$ is immediate from (2.16),  
using the fact that $p^2\mid R_1$, and so $p^6\mid p^2R_1^2$. 
Since from (2.14) we have $p^2R_2\equiv -2pR_1-\frac{p^3}{3}
\,\,(\bmod{\,\,p^6})$, inserting this into 
the first congruence in Corollary 1.4, we immediately obtain
 \begin{equation*}  
{2p-1\choose p-1}\equiv 1 +2p\sum_{k=1}^{p-1}\frac{1}{k}
+\frac{2p^3}{3}\sum_{k=1}^{p-1}\frac{1}{k^3}\pmod{p^6}, 
    \end{equation*} 
which is just the second congruence in Corollary 1.4.

A calculation shows that both congruences are also satisfied 
for $p=7$, and the proof is completed.
\end{proof}

\begin{proof}[Proof of Corollary 1.5.]
Let $p\ge 7$ be any prime. By Corollary 1.4, we have  
${2p-1\choose p-1}\equiv  1-2pR_1-2p^2R_2\,\,(\bmod{\,\,p^5})$.
Substituting  into this $-pR_{2}\equiv 2R_{1}\,\,(\bmod{\,\,p^4})$
(Lemma 2.4), we obtain 
   \begin{equation*}  
{2p-1\choose p-1}\equiv 1 +2p\sum_{k=1}^{p-1}\frac{1}{k}
\pmod{p^5}, 
    \end{equation*} 
as desired.\end{proof}

\section{Proof of  Corollary 1.3} 
As noticed in the Introduction, in the proof of Corollary
1.3, we wiil apply a method of Helou and Terjanian \cite{HT} 
based on Kummer type congruences.

\begin{lemma}
Let  $p$  be a prime, 
and let $m$ be any even positive integer.
Then the denominator $d_{m}$ of the Bernoulli number $B_{m}$
written in reduced form,  is given by
   $$
d_{m}=\prod_{p-1\mid m}p,
  $$
where the product is taken over those primes $p$ such that $p-1$ divides $m$.
\end{lemma}

\begin{proof} The assertion is an immediate 
consequence of the von Staudt-Clausen theorem 
(eg. see \cite{I},  p. 233, Theorem 3) which asserts
that  $B_{m}+\sum_{p-1\mid m}1/p$ is an integer for all even $m$, where 
the summation is over all primes $p$ such that 
$p-1$ divides  $m$.   
\end{proof} 

For a prime $p$  and a positive integer
$n$,  we denote 
  $$
R_n(p)=R_n=\sum_{k=1}^{p-1}\frac{1}{k^n}\quad
{\rm and}\quad P_n(p)=P_n=\sum_{k=1}^{p-1}k^n.
  $$
\begin{lemma}  {\rm (\cite{HT}, p. 8).}  Let $p$ be a prime greater than
$5$, and let $n,r$ be positive integers. Then
  \begin{equation}\label{cong3.1}
P_n(p)\equiv\sum_{s-{\rm ord}_p(s)\le r}
\frac{1}{s}{n\choose s-1}p^sB_{n+1-s}\pmod{p^r},
   \end{equation}
where ${\rm ord}_p(s)$ is the largest power of $p$ dividing $s$,
and the summation is taken over all integers $1\le s\le n+1$
such that $s-{\rm ord}_p(s)\le r$.
\end{lemma} 

The following result is well known as the Kummer congruences.

\begin{lemma} {\rm (\cite{I})}.  Suppose that $p\ge 3$ is a prime and 
$m$, $n$, $r$ are positive integers such that $m$ and $n$ are even,
$r\le n-1\le m-1$ and $m\not\equiv 0\,\,(\bmod{\,\, p-1})$.
If $n\equiv m\,\,(\bmod{\,\,\varphi(p^r)})$, where $\varphi(p^r)=p^{r-1}(p-1)$
is the Euler's totient function, then
 \begin{equation}\label{cong3.2}
\frac{B_m}{m}\equiv\frac{B_n}{n}\pmod{p^r}.
  \end{equation}
\end{lemma} 

The following congruences are also due to Kummer.

\begin{lemma} {\rm (\cite{K}; also see \cite{HT}, p. 20)}.  Let $p\ge 3$ 
be a prime and let $m$, $r$ be positive integers such that $m$ is even, 
$r\le m-1$ and $m\not\equiv 0\,\,(\bmod{\,\, p-1})$. Then
 \begin{equation}\label{cong3.3}
\sum_{k=0}^{r}(-1)^k{m\choose k}\frac{B_{m+k(p-1)}}{m+k(p-1)}\equiv 0
\pmod{p^r}.
 \end{equation}
\end{lemma}

\begin{lemma} For any prime $p\ge 11$, we have
\begin{itemize}
\item[(i)] $\displaystyle R_1(p)\equiv -\frac{p^2}{2}B_{p^4-p^3-2}
-\frac{p^4}{4}B_{p^2-p-4}+\frac{p^5}{6}B_{p-3}+\frac{p^5}{20}B_{p-5}
\,\,(\bmod{\,\, p^6}).\qquad$
\item[(ii)] $\displaystyle R_1^2(p)\equiv \frac{p^4}{9}B_{p-3}^2
\,\,(\bmod{\,\, p^5}).\qquad$
\item[(iii)] $\displaystyle R_2(p)\equiv 
pB_{p^4-p^3-2}+p^3B_{p^4-p^3-4}\,\,(\bmod{\,\, p^5}).$ 
\end{itemize}
  \end{lemma} 

\begin{proof}
If $s$ is a positive integer such that 
${\rm ord}_p(s)=e\ge 1$, then for $p\ge 11$ holds $s-e\ge p^e-e\ge 10$.
This shows that the condition $s-{\rm ord}_p(s)\le 6$
implies that ${\rm ord}_p(s)=0$, and thus, for such a $s$ must be $s\le 6$.
Therefore
 \begin{equation}\label{cong3.4}
P_n(p)\equiv\sum_{s=1}^6\frac{1}{s}{n\choose s-1}
p^sB_{n+1-s}\pmod{p^6}\quad {\rm for}\quad n=1,2,\ldots. 
 \end{equation}
By Euler's theorem \cite{HW}, for $1\le k\le p-1$, and positive integers
$n,e$ we have $1/k^{\varphi(p^e)-n}\equiv k^{n}\,\,(\bmod{\,\, p^e})$,
where   $\varphi(p^e)=p^{e-1}(p-1)$ is the Euler's totient function. 
Hence, $R_{\varphi(p^e)-n}(p)\equiv P_n(p)\,\,(\bmod{\,\, p^e})$.
In particular, if $n=\varphi(p^6)-1=p^5(p-1)-1$, then 
by Lemma 3.1, $p^6\mid p^6B_{p^5(p-1)-6}$ for each prime $p\ge 11$.
Therefore, using the fact that 
$B_{p^5(p-1)-1}=B_{p^5(p-1)-3}=B_{p^5(p-1)-5}=0$, (12) yields 
    \begin{eqnarray*}   
R_1(p)&\equiv & P_{p^5(p-1)-1}(p)\equiv \frac{1}{2}(p^5(p-1)-1)p^2
B_{p^5(p-1)-2}\\
&& +\frac{1}{4}\frac{(p^5(p-1)-1)(p^5(p-1)-2)(p^5(p-1)-3)}{6}p^4
B_{p^5(p-1)-4}\pmod{p^6},
       \end{eqnarray*}  
whence we have
  \begin{equation}\label{cong3.5}
R_1(p)\equiv -\frac{p^2}{2}B_{p^6-p^5-2}
-\frac{p^4}{4}B_{p^6-p^5-4}\pmod{p^6}.
  \end{equation}
By the Kummer congruences \eqref{cong3.2} from Lemma 3.3, we have
 $$
B_{p^6-p^5-2}\equiv
\frac{p^6-p^5-2}{p^4-p^3-2}B_{p^4-p^3-2}\equiv
\frac{2B_{p^4-p^3-2}}{p^3+2}\equiv
\left(1-\frac{p^3}{2}\right)B_{p^4-p^3-2}\pmod{p^4}.
 $$
Substituting this into \eqref{cong3.5}, we obtain
 \begin{equation}\label{cong3.6}
R_1(p)\equiv -\frac{p^2}{2}B_{p^4-p^3-2}+\frac{p^5}{4}B_{p^4-p^3-2}
-\frac{p^4}{4}B_{p^6-p^5-4}\pmod{p^6}.
  \end{equation}
Similarly, we have 
 $$
B_{p^4-p^3-2}\equiv
\frac{p^4-p^3-2}{p-3}B_{p-3}\equiv \frac{2}{3}B_{p-3}\pmod{p}
 $$
and
 $$
B_{p^6-p^5-4}\equiv
\frac{p^6-p^5-4}{p^2-p-4}B_{p^2-p-4}\equiv \frac{4B_{p^2-p-4}}{p+4}
\equiv \left(1-\frac{p}{4}\right)B_{p^2-p-4}\pmod{p^2}.
 $$
Substituting the above two congruences into \eqref{cong3.6}, we get
 \begin{equation}\label{cong3.7}
R_1(p)\equiv -\frac{p^2}{2}B_{p^4-p^3-2}+\frac{p^5}{6}B_{p-3}
-\frac{p^4}{4}B_{p^2-p-4}+\frac{p^5}{16}B_{p^2-p-4}\pmod{p^6}.
  \end{equation}
Finally, since
 $$
B_{p^2-p-4}\equiv\frac{p^2-p-4}{p-5}B_{p-5}\equiv
\frac{4}{5}B_{p-5}\pmod{p},
 $$
the substitution of the above 
congruence  into \eqref{cong3.7} immediately gives the congruence (i).

Further, \eqref{cong3.7} immediately gives
  \begin{equation}\label{cong3.8}
R_1(p)^2\equiv \frac{p^4}{4}B_{p^4-p^3-2}^2\pmod{p^5}.
  \end{equation}
Again by the Kummer congruences \eqref{cong3.2} from Lemma 3.3, we have
 $$
B_{p^4-p^3-2}\equiv\frac{p^4-p^3-2}{p-3}B_{p-3}\equiv
\frac{2}{3}B_{p-3}\pmod{p}.
 $$
Substituting this into \eqref{cong3.7}, we immediately obtain
the congruence (ii).

In order to prove the congruence (iii),
note that if $n-3\not\equiv 0\,\,(\bmod{\,\, p-1})$,
then by Lemma 3.1, for even $n\ge 6$ holds $p^5\mid p^5B_{n-4}$,
and we known that $B_{n-1}=B_{n-3}=0$ for such a $n$.
Therefore, reducing the modulus in \eqref{cong3.4} to  $p^5$, 
and using the same argument as in the begin of the proof of (i), for 
all even $n\ge 2$  holds
 \begin{equation}\label{cong3.9}
P_n(p)\equiv pB_n+\frac{p^3}{6}n(n-1)B_{n-2}\pmod{p^5}.
 \end{equation}
In particular, for $n=p^4-p^3-2$  
and using $P_{\varphi(p^4)-2}(p)\equiv R_2(p)\,\,(\bmod{\,\, p^4})$,
\eqref{cong3.9} reduces to
     \begin{equation}\label{cong3.10}
R_2(p)\equiv P_{p^4-p^3-2}(p)\equiv pB_{p^4-p^3-2}+
p^3B_{p^4-p^3-4}\pmod{p^5}.
   \end{equation}
This completes the proof.
  \end{proof}
\begin{proof}[Proof of Corollary 1.3]
The congruence of Corollary 1.3  follows directly
by substituting the congruences (i), (ii) and (iii) of Lemma 3.5 into 
the congruence \eqref{cong1.4} of Theorem 1.1.
  \end{proof}

\end{document}